\documentclass[11pt,reqno]{amsart}
\usepackage[dvips]{psfrag}
\usepackage{graphicx}
\usepackage{color,colortbl}
\usepackage{amsthm,amssymb,amsmath,amsfonts,amscd}
\usepackage{tikz}
\usetikzlibrary{arrows}

\def\dis{\displaystyle}

\newcommand{\lf}{\mu_f}

\theoremstyle{plain}
\newtheorem {theorem} {Theorem}[section]
\newtheorem {proposition}[theorem]{Proposition}
\newtheorem {corollary} [theorem]{Corollary}
\newtheorem {lemma}[theorem]{Lemma}
\newtheorem{maintheorem}{Theorem}

\theoremstyle{definition}

\newtheorem {remark}[theorem]{Remark}

\newcommand{\Q}{\mathcal{Q}}

\newcommand{\A}{\mathcal{A}}

\newcommand{\Leb}{\operatorname{Leb}}

\newcommand{\HD}{\operatorname{HD}}
\newcommand{\BD}{\operatorname{c}}
\newcommand{\diam}{\operatorname{diam}}
\newcommand{\Jac}{\operatorname{Jac}}
\newcommand{\length}{\operatorname{length}}

\newcommand{\CC}{{\mathcal C}}

\newcommand{\OO}{{\mathcal O}}

\newcommand{\SSS}{{\mathcal S}}

\begin{document}

\title[Continuity of Hausdorff Dimension at Hopf Bifurcation]
{Continuity of Hausdorff Dimension at Hopf Bifurcation}

\author[V. Horita and O. Raizzaro]
{Vanderlei Horita$^1$ and Oyran Raizzaro$^2$}

\address{$^{1}$ S\~ao Paulo State University (UNESP), Institute of Biosciences, Humanities and Exact Sciences (Ibilce), Campus S\~ao Jos\'e do Rio Preto.
Rua C. Colombo, 2265, CEP 15054--000 S. J. Rio Preto, S\~ao Paulo,
Brazil.}

\address{$^2$ State University of Mato Grosso do Sul (UEMS), Departament of Mathematics,
Rua W. Hubacher, 138, CEP 79750-035 Nova Andradina, Mato Gros\-so do Sul, Brazil.}


\begin{abstract}
We investigate the continuity of the Hausdorff dimension and the box dimension (limit
capacity) of non-hyperbolic repellers of diffeomorphisms derived from
transitive Anosov diffeomorphisms through a Hopf bifurcation studied by
Horita-Viana in \cite{HV05}.
Here, we extend their work showing that both dimensions are continuous at the
bifurcation parameter.
In the proof, we consider \emph{maps with holes} introduced
by Horita-Viana in \cite{HV01} and further developed by Dysman in \cite{Dys05}, relating the
Hausdorff dimension with the volume of the hole.
\end{abstract}

\maketitle

\let\thefootnote\relax
\let\thefootnote\relax
\let\thefootnote\relax\footnote{Work partially supported by Capes, CNPq,
FAPESP, and PRONEX.}

\section{Introduction}

Hopf bifurcations arise in a wide range of physical phenomena, including fluid dynamics, chemical reactions, population dynamics, electrical circuits, and cellular biology. This subject has been extensively studied; see, for instance, \cite{ AY78,JS72, MM76,  Sm72, So73b}.

The continuity of Hausdorff dimension remains a delicate and largely unresolved problem in smooth dynamics. Several works \cite{DV89,Ma85, McMa83, PV88, Ra92} investigate its behavior for invariant sets associated with one-parameter families of diffeomorphisms near bifurcation parameters. In particular, \cite{DV89} constructs examples of families undergoing non-degenera-te saddle-node bifurcations for which the Hausdorff dimension is continuous in some cases and discontinuous in others. To the best of our knowledge, a general criterion ensuring continuity is still unknown. Estimating the Hausdorff dimension of a set is highly non-trivial, especially when determining lower bounds, and even upper bounds can be difficult to obtain; see, for example, \cite{BDV95, Dys05, HV01, HV05, SS99, Ur86}.

In this work, we prove that the Hausdorff dimension of the attracting sets associated with the families of DA-diffeomorphisms introduced in \cite{HV05} converges to the dimension of the ambient manifold. We begin by recalling their construction.

Let $\mathbb{T}^3$ denote the three-dimensional torus, and consider a linear Anosov diffeomorphism $G$ on $\mathbb{T}^3$ with one real eigenvalue $\lambda \in (0,1)$ and a pair of complex conjugate eigenvalues $\sigma e^{\pm i\alpha}$, with $\sigma > 3$ (this condition is technical and not optimal in \cite{HV05}). Assume that $k\alpha \notin 2\pi\mathbb{Z}$ for $k=1,2,3,4$. In cylindrical coordinates $(\rho,\theta,z)$ near the fixed point $(0,0,0)$, the map takes the form
\[
G(\rho,\theta,z) = (\sigma \rho, \theta + \alpha, \lambda z).
\]

From $G$, we construct a family of diffeomorphisms $\hat{g}_\mu$, $\mu \in [-1,1]$, undergoing a Hopf bifurcation at $\mu=0$. To this end, we introduce a $C^\infty$ function $\Phi(\mu,w,z)$ defined on $[-1,1]^3$ (see Figure~\ref{f.phi}), satisfying conditions $(C_1)$–$(C_4)$ below for some constant $C_0>0$ and sufficiently small $\delta_0>0$:

 \begin{itemize}
 	\item[($C_1$)] $\Phi(\mu,0,0) = 1 - \mu \le\Phi(\mu,w,z)$
 	for all $w\ge 0$.
 	\item[($C_2$)] $\Phi(\mu,w,z) = \sigma$ when
 	either $w \ge \delta_0$ or $|z| \ge \delta_0$\,.
 	\item[($C_3$)] $0 < \partial_w \Phi(\mu,w,z) \le C_0/\delta_0$
 	when $0\le w < \delta_0$ and $|z| < \delta_0$\,.
 	\item[($C_4$)] There exist  $\sigma_1\in(1,\sigma)$ and
 	$\delta_1\in(0,\delta_0)$ such that $\Phi(\mu,w,z)>\sigma_1$
 	for all $w\ge\delta_1$, and
 	$\partial_w \Phi(\mu,w,z)\ge \partial_w \Phi(\mu,0,0)$
 	for all $w\in[0,\delta_1]$.
 \end{itemize}

 	\begin{figure}[phtb]
 	\begin{center}
 		\includegraphics[height=4.5cm]{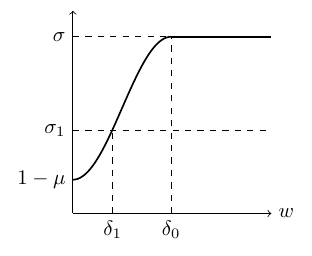}
 		\caption{\label{f.phi}Graph of $\Phi(\mu,\cdot,z)$}
 	\end{center}
 \end{figure}

We choose $\delta_0$ sufficiently small so that the domain
\[
\{(\rho,\theta,z) : \rho^2 \le \delta_0,\ |z| \le \delta_0\}
\]
is contained in a neighborhood $V$ of the origin whose closure lies inside a rectangle of a Markov partition for $G$.

We then deform $G$ inside $V$ to obtain a one-parameter family $\hat{g}_\mu$ that coincides with $G$ outside $V$ and is given in cylindrical coordinates by
\begin{equation}\label{e.local_form}
\hat{g}_\mu(\rho,\theta,z)
= \big(\Phi(\mu,\rho^2,z)\rho,\ \theta+\alpha,\ \lambda z\big).
\end{equation}
The origin remains a fixed point for all $\mu$, and the family $(\hat{g}_\mu)_\mu$ undergoes a generic Hopf bifurcation at $\mu=0$ (see Proposition~\ref{p.existehopf}).

We consider families $(g_\mu)_\mu$, $\mu \in [-1,1]$, in a $C^5$-neighborhood of $(\hat{g}_\mu)_\mu$. Since the unfolding of a generic Hopf bifurcation is an open property in the $C^5$ topology (see \cite{MM76}, \cite[Section~2.2]{HV05}, and Section~\ref{ss.hopf}), any such family admits a unique curve of fixed points $p_\mu$ near the origin and undergoes a Hopf bifurcation at some parameter $\mu_*$ close to $0$. As $\mu$ crosses $\mu_*$, the system transitions from an Anosov diffeomorphism to a DA-type diffeomorphism: the fixed point $p_\mu$, initially a hyperbolic saddle for $\mu < \mu_*$, becomes an attracting point for $\mu > \mu_*$. A representative scenario is shown in Figure~\ref{fig1}.

\begin{figure}[phtb]
	\begin{center}
		\includegraphics[scale=0.4]{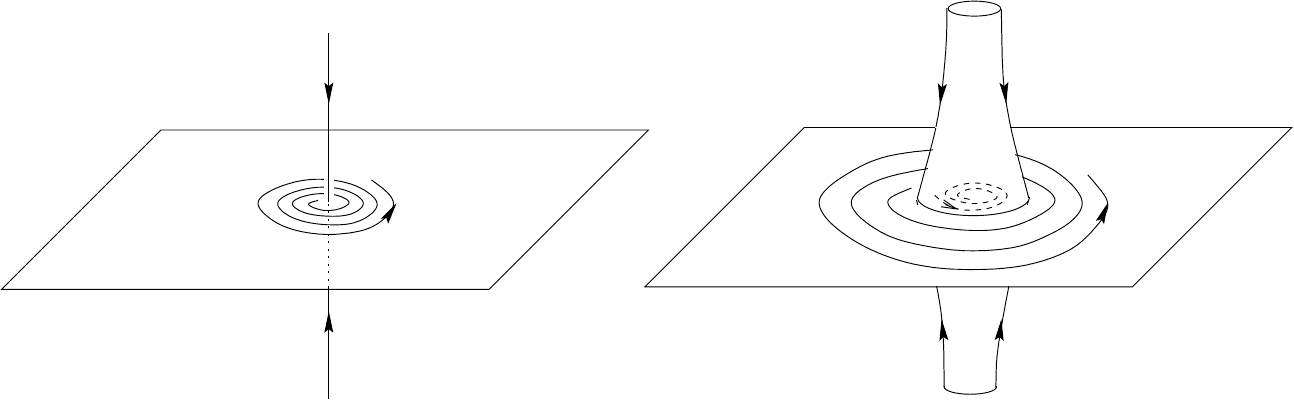}
		\caption{Hopf bifurcation}\label{fig1}
	\end{center}
	\end{figure}

For $\mu > \mu_*$, the non-wandering set consists of the sink $p_\mu$ and a non-hyperbolic repeller
\[
\Lambda_\mu := \mathbb{T}^3 \setminus W^s(p_\mu).
\]
Thus, the complement of the basin of attraction of $p_\mu$ forms a global repeller. A precise description of these families is given in Section~\ref{ss.hopf}. Horita and Viana \cite{HV05} showed that for all $\mu > \mu_*$, both the Hausdorff and box dimensions of $\Lambda_\mu$ are strictly less than $3$, leaving open the question of their continuity. Our main result resolves this issue.

\begin{maintheorem}
\label{t.thmB}
Let $(g_\mu)_\mu$ be a family of diffeomorphisms in a $C^5$-neighbor-hood of $(\hat{g}_\mu)_\mu$. Then
\[
\lim_{\mu \to \mu_*} \HD(\Lambda_\mu) = 3.
\]
In particular, $\displaystyle \lim_{\mu \to \mu_*} \BD(\Lambda_\mu) = 3$.
\end{maintheorem}

Let us recall the definitions of the Hausdorff dimension and the box dimension.
The \emph{Hausdorff dimension} of a compact metric space $X$
is the unique real number $\HD(X)$ such that
$m_\alpha(X) = \infty$ for any $\alpha < \HD(X)$ and
$m_\alpha(X) = 0$ for any $\alpha > \HD(X)$,
where $m_\alpha$ is the Hausdorff $\alpha$-measure, defined by
\begin{align*}
	m_\alpha(X)=\lim_{\varepsilon \rightarrow 0}
	\inf \bigg\{
	\sum_{U\in \mathcal{U}}(\diam\ U)^\alpha \colon &  \mathcal{U} \text{ is an open
		covering of } X \text{ with } \\ & \qquad \diam\ U \leq \varepsilon
	\text{ for all } U \in \mathcal{U} \bigg\}.
\end{align*}
The \emph{box dimension} (or \emph{limit capacity}) $\BD(X)$ is defined by:
$$
\BD(X)=\limsup_{\varepsilon\to 0}
\frac{\log n(X,\varepsilon)}{|\log\varepsilon|}\,,
$$
where $n(X,\varepsilon)$ is the smallest number of
$\varepsilon$-balls needed to cover $X$.
One always has $\HD(X)\le \BD(X)$.

Our result extends naturally to higher dimensions: for families derived from Anosov diffeomorphisms on $\mathbb{T}^D$ with two-dimensional expanding direction
\[
\lim_{\mu \to \mu_*} \HD(\Lambda_\mu)
= \lim_{\mu \to \mu_*} \BD(\Lambda_\mu)
= D.
\]

Related families have been studied since Carvalho \cite{Car93}, where non-hyperbolic attractors were constructed and shown to support unique SRB measures. This was later generalized by Bonatti and Viana \cite{BoV00} to a broader class of partially hyperbolic attractors. In particular, they observed that such attractors have empty interior and dense strong-unstable leaves, implying transitivity.

For $\mu$ close to $0$, the diffeomorphisms $g_\mu$ are partially hyperbolic: the tangent bundle admits a dominated splitting $TM = E^{ss}_\mu \oplus E^{cu}_\mu$,
where $E^{ss}_\mu$ is uniformly contracting. The system also admits invariant foliations: a strong-stable foliation $\mathcal{F}^{ss}_\mu$ and a center-unstable foliation $\mathcal{F}^{cu}_\mu$, the latter being close to the unstable foliation of $\hat{g}_{-1}$. The strong-stable foliation is of class $C^{1+\delta}$; see Section~\ref{ss.partialhyperbolicity}.

The proof of Theorem~\ref{t.thmB} is a consequence of the continuity of the Hausdorff dimension of the intersection of $\Lambda_\mu$ with any center-unstable leaf of the central foliation $\mathcal{F}^c_\mu$ (see Section~\ref{s.proofs}).

\begin{maintheorem}
\label{t.thmC}
Let $(g_\mu)_\mu$ be as above. Then, for every $x$,
\[
\lim_{\mu \to \mu_*} \HD\big(\Lambda_\mu \cap \mathcal{F}^c_\mu(x)\big) = 2.
\]
In particular, $\displaystyle \lim_{\mu \to \mu_*} \BD\big(\Lambda_\mu \cap \mathcal{F}^c_\mu(x)\big) = 2$.
\end{maintheorem}

The proof of Theorem~\ref{t.thmC} uses an abstract class of dynamical systems, called \emph{maps with holes}, induced by $g_\mu$ on the leaves of $\mathcal{F}^c_\mu$. These are analogous to the systems studied in \cite{HV01, Dys05}, and satisfy suitable decay estimates on the volume of \emph{bad cylinders}. Precise definitions and conditions are given in Section~\ref{ss.maps}.

\subsection{Maps with Holes}
\label{ss.maps}
The works \cite{HV01} and \cite{Dys05} study an abstract class of dynamical systems known as \emph{maps with holes}. In the present paper, we employ a closely related framework.

Let $M$ be a $d$-dimensional Riemannian manifold, $d \geq 1$. A map $f \colon M \to M$ is called a \emph{map with holes} if it satisfies the following conditions:

\begin{itemize}
\item[$(\mathcal{A}_1)$] There exist domains $R_0, \dots, R_m$ with finite inner diameter and pairwise disjoint interiors such that, for each $i = 0, \dots, m$, the restriction $f|_{R_i}$ is a $C^{1+\varepsilon}$ diffeomorphism onto a domain $W_i$ satisfying:
\end{itemize}
\begin{itemize}
\item $W_i \supset \bigcup_{j=0}^m R_j$;
\item the set $H_i := W_i \setminus \bigcup_{j=0}^m R_j$, called \emph{hole}, has nonempty interior;
\item the boundary $\partial R_i$ has box dimension strictly less than $d$.
\end{itemize}

Here, a \emph{domain} is a compact path-connected subset of $M$. The condition of finite inner diameter means that there exists $L > 0$ such that any two points in $R_i$ can be joined by a piecewise $C^1$ curve contained in $R_i$ of length at most $L$. We say that $f$ is of class $C^{1+\varepsilon}$ if it is differentiable and $|\det Df|$ is $\varepsilon$-Hölder continuous.

\begin{remark}
\label{r.weakmarkov}
It is sufficient to assume that each $W_i$ contains the union of some subcollection of the sets $R_j$, and that forward iterates of each $R_i$ eventually intersect a hole; that is, for some $k$, the image contains a domain $W_k$ such that $H_k = W_k \setminus \bigcup_j R_j$ has nonempty interior. See \cite[Remark~1]{HV01}.
\end{remark}

The \emph{repeller} of a map with holes $f$ is the set
\[
\Lambda_f = \{x \in M : f^n(x) \in \bigcup_{i=0}^m R_i \text{ for all } n \geq 0\},
\]
that is, the set of points whose forward orbits never enter any hole.

	\begin{figure}[phtb]
	\begin{center}
		\includegraphics[height=4.0cm]{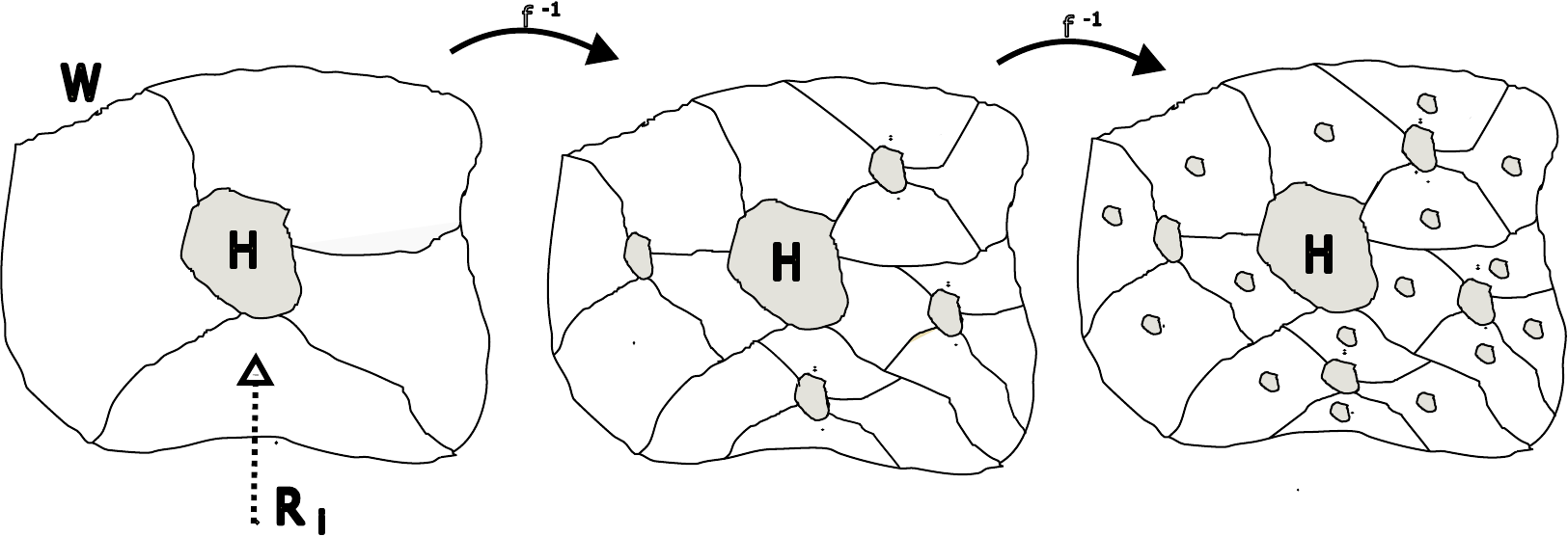}
		\caption{The repeller of a piecewise smooth map with holes.}
	\end{center}
\end{figure}

For $n \geq 1$ and $\alpha_1, \dots, \alpha_n \in \{0,1,\dots,m\}$, an \emph{$n$-cylinder} is any set of the form
\[
C(\alpha_1,\dots,\alpha_n)
= R_{\alpha_1} \cap f^{-1}(R_{\alpha_2}) \cap \cdots \cap f^{-n+1}(R_{\alpha_n}).
\]
For each $n$, the collection of $n$-cylinders forms a cover of $\Lambda_f$.

Let $f$ be a map with holes. We call the \emph{average least expansion} of an $n$-cylinder $C(\alpha_1,\dots,\alpha_n)$ by
\begin{equation}
\label{eq.phi}
\phi_n(\alpha_1,\dots,\alpha_n)
= \frac{1}{n} \sum_{j=1}^n
\inf_{x \in C_j} \log \|Df^{-1}(f^j(x))\|^{-1},
\end{equation}
where $C_j = C(\alpha_1,\dots,\alpha_j)$. For convenience, for all $j \ge i$ we interpret $Df^{-i}(f^j(y))$ as the inverse of $Df^{i}(f^{j-i}(y))$ .

If $\phi_n(\alpha_1,\dots,\alpha_n) > c_0 > 0$, then for every $x \in C(\alpha_1,\dots,\alpha_n)$, 
$\prod\limits_{j=1}^n \|Df^{-1}(f^j(x))\| \le e^{-c_0 n}$,
and hence $\|Df^{-n}(f^n(x))\| \le e^{-c_0 n}$. Thus, $Df^n$ expands every tangent vector at $x$ by at least $e^{c_0 n}$.

Let
\begin{equation}
\label{eq.hf}
H_f = \bigcup_{i=0}^m H_i
\quad \text{and} \quad
\mu_f = \Leb(H_f),
\end{equation}
where $\Leb$ denotes Lebesgue measure. Given $c_0 > 0$, define $B_n(c_0 \mu_f)$ as the set of points contained in some $n$-cylinder $C(\alpha_1,\dots,\alpha_n)$ such that
\[
\phi_j(\alpha_1,\dots,\alpha_j) \le c_0 \mu_f,
\quad \text{for all } j = 1,\dots,n.
\]
That is,
\begin{equation}
\label{eq11}
\begin{aligned}
B_n(c_0\lf) = \Big\{x \in M \colon x \in C(\alpha_1, \dots ,\alpha_n) \text{ with } &
\phi_j(\alpha_1, \dots ,\alpha_j) \leq c_0\lf , \\
&{\rm\text{for all}}\; j \in \{1,\dots,n\}\Big \}.
\end{aligned}
\end{equation}
We note that $B_n(c_0\lf)$ is a union of $n$-cylinders and we refer to such cylinders as \emph{bad $n$-cylinders}.

\medskip

We assume the following condition on the decay of bad cylinders:

\begin{itemize}
\item[$(\mathcal{A}_2)$] There exist positive constants $n_0, c_0, \mu_0$ and a function
$\zeta \colon \mathbb{N} \times [0,\Leb(M)) \to [0,\infty)$ such that for all $n \ge n_0$ and $0 \le \mu_f < \mu_0$ 
\[
\Leb\big(B_n(c_0 \mu_f)\big) \le \zeta(n,\mu_f),
\]
Furthermore, for each fixed $n \ge n_0$, $\zeta(n,\mu_f) \to 0$ as $\mu_f \to 0$.
\end{itemize}

We denote $f_j = f|R_j$, $j \in \{0, \cdots , m\}$ and we call \emph{inverse branch} of $f$ any map of the form $f^{(-1)} = f_j^{-1}$.

\begin{maintheorem}
\label{t.thmA}
Let $f \colon M \to M$ be a map with holes on a $d$-dimensional Riemannian manifold, $d \ge 1$, satisfying $(\mathcal{A}_1)$ and $(\mathcal{A}_2)$. For some constant $S> 0$, assume that $\sup \|Df^{(-1)}\| \le S$ for any inverse branch of $f$.
Then, for every $\varepsilon > 0$, there exists $\delta = \delta(\varepsilon, S, m, c_0, n_0, \mu_0) > 0$ such that, if $\Leb(H_f) < \delta$, then
\[
\HD(\Lambda_f) > d - \varepsilon.
\]
In particular, $\BD(\Lambda_f) > d - \varepsilon$.
\end{maintheorem}

Condition $(\mathcal{A}_2)$ controls the decay of the volumes of bad cylinders as a function of both the length $n$ and the size of the hole. In contrast, Dysman \cite{Dys05} assumes exponential decay in $n$, uniformly with respect to the hole size.

\begin{remark}
If $(\mathcal{A}_2)$ is strengthened by requiring that $\zeta(n,\mu_f) \to 0$ as $n \to \infty$ for each fixed $\mu_f > 0$, then one recovers the classical conclusion that $\HD(\Lambda_f)$ and $\BD(\Lambda_f)$ are strictly less than $d$; see \cite{HV01, Dys05}. 
However, if these dimensions are equal to $d$ then Theorem~\ref{t.thmA} would follow trivially.
\end{remark}

Dysman \cite[Corollary~5]{Dys05} also studies the continuity of fractal dimensions for families of maps with holes. However, her results do not apply to the families $(g_\mu)$ considered here. Specifically, condition $(NU_3)$ in \cite{Dys05} fails in our setting: in her framework, the volume of bad $n$-cylinders decays uniformly in $n$, independently of the hole size, whereas in our case the decay depends on both $n$ and $\mu_f$.

Section~\ref{s.def} reviews the construction of the families $(g_\mu)$ and their main properties. In Section~\ref{s.proofs}, we derive Theorems~\ref{t.thmB} and \ref{t.thmC} from Theorem~\ref{t.thmA}. The proof of Theorem~\ref{t.thmA} is given in Section~\ref{s.mwh}, while Section~\ref{s.lyap} establishes condition $(\mathcal{A}_2)$ for the families induced by $(g_\mu)$.

We conclude by noting that our results also apply to certain families studied by Díaz and Viana \cite{DV89}, including examples exhibiting both continuity and discontinuity of fractal dimensions at bifurcation parameters.

\section{Definitions and basic results}
\label{s.def}

In this section, we recall the constructions and definitions mentioned above and review some results necessary for the proofs.

\subsection{Partial hyperbolicity and invariant foliations}
\label{ss.partialhyperbolicity}

We start with the following property:
\begin{itemize}
\item The diffeomorphisms $g_\mu$ are partially hyperbolic for every $\mu$ sufficiently close to zero: there exists a $Dg_\mu$-invariant dominated splitting $TM = E_\mu^{ss} \oplus E_\mu^{cu}$, where $E_\mu^{ss}$ is uniformly contracting.
\end{itemize}

Since partial hyperbolicity is an open property in the $C^1$ topology, it suffices to verify it for the reference family $\hat{g}_\mu$ for $\mu$ small. This is established in the following proposition.

\begin{proposition}\label{p.bundle}
Fix $\varepsilon > 0$ and assume that $\delta_0$ is sufficiently small. Then, for every $\mu$ in a neighborhood of zero, there exists a $D\hat{g}_\mu$-in\-var\-i\-ant splitting $TM = E_\mu^{ss} \oplus E_\mu^{cu}$, and constants $K_s, K_c, K > 0$ such that:
\begin{itemize}
\item[(a)] $E_\mu^{cu}$ coincides with the $(\rho,\theta)$-plane;
\item[(b)] $\| D\hat{g}_\mu^n | E_\mu^{ss} \| \le K_s \lambda^n$ for all $n \ge 1$ (uniform contraction);
\item[(c)] $\| D\hat{g}_\mu^n | E_\mu^{cu} \| \le K_c (\sigma + \varepsilon)^n$ for all $n \ge 1$;
\item[(d)] $\| D\hat{g}_\mu^{-n} | E_\mu^{cu} \| \le K (1 - \mu)^{-n}$ for all $n \ge 1$;
\item[(e)] $\| D\hat{g}_\mu^{-n} | E_\mu^{cu} \| \cdot \| D\hat{g}_\mu^n | E_\mu^{ss} \| < \tfrac{1}{2}$ for some $n \ge 1$ (domination).
\end{itemize}
\end{proposition}

\begin{proof}
See \cite[Proposition 2.1]{HV05}.
\end{proof}

\begin{corollary}
\label{c.foliation}
There exists $n \ge 1$ such that
\[
\sup_M \big\|(D g_\mu^n | E^{cu}_\mu)^{-1}\big\| \cdot
\big\|D g_\mu^n | E^{ss}_\mu \big\| \cdot
\big\|D g_\mu^n | E^{cu}_\mu \big\| < 1
\]
for every $\mu$ sufficiently close to zero.
\end{corollary}

\begin{proof}
See \cite[Corollary 2.2]{HV05}, together with the continuity of the invariant subbundles $E_\mu^{ss}$ and $E_\mu^{cu}$ in the $C^1$ topology.
\end{proof}

\begin{remark}\label{r.expansion}
Since $\hat{g}_\mu$ coincides with $G$ outside $V$, it expands the center-unstable bundle $E^{cu}_\mu$ on $M \setminus V$:
\[
\big\| (D\hat{g}_\mu(x) | E^{cu}_\mu(x))^{-1} \big\| \le \sigma^{-1}
\quad \text{for all } x \in M \setminus V.
\]
The same property holds for any $g_\mu$ sufficiently close to $\hat{g}_\mu$ (possibly after slightly decreasing $\sigma > 3$).
\end{remark}

We now comment on the significance of Corollary~\ref{c.foliation}. By the theory of Hirsch–Pugh–Shub~\cite{HP70,HPS77}, this condition ensures that, throughout the bifurcation, the map $g_\mu$ admits invariant foliations: a \emph{center-unstable} (or \emph{central}) foliation $\mathcal{F}^c_\mu$, whose tangent bundle $E^{cu}_\mu$ is everywhere contained in the cone field $C^{cu}_\mu$, and a \emph{strong-stable} foliation $\mathcal{F}^{ss}_\mu$, whose tangent bundle $E^{ss}_\mu$ is everywhere contained in the cone field $C^{ss}_\mu$ and whose leaves are uniformly contracted by $g_\mu$. Moreover, the strong-stable foliation enjoys additional regularity:

\begin{corollary} \label{c.folheacao}
There exists $\beta > 0$ such that, for every $\mu$ sufficiently close to zero, the diffeomorphism $g_\mu$ admits a $C^{1+\beta}$ strong-stable foliation $\mathcal{F}^{ss}_\mu$ of codimension $2$, invariant under $g_\mu$, whose tangent bundle is contained in the cone field $C^{ss}_\mu$. In particular, its leaves are uniformly contracted by $g_\mu$.
\end{corollary}

Furthermore, again by \cite{HP70,HPS77}, the center-unstable foliation $\mathcal{F}^c_\mu$ is topologically conjugate to the unstable foliation $\mathcal{F}^u$ of $G$: there exists a homeomorphism mapping leaves of one foliation to leaves of the other. This follows from the fact that $\mathcal{F}^c_\mu$ remains normally contracting throughout the isotopy (together with the plaque expansiveness condition). In particular, all center-unstable leaves are dense in $M$.

\subsection{Hopf bifurcation}
\label{ss.hopf}

We next show that any one-parameter family of diffeomorphisms $(g_\mu)_\mu$ that is $C^5$-close to $(\hat{g}_\mu)_\mu$ exhibits the following behavior:

\begin{itemize}
\item There exists a curve $(p_\mu)_\mu$ of fixed points of $(g_\mu)_\mu$, and a parameter value $\mu_*$ close to zero such that $p_\mu$ is a hyperbolic saddle for $\mu < \mu_*$, undergoes a generic Hopf bifurcation at $\mu = \mu_*$, and becomes an attracting fixed point for $\mu > \mu_*$. Moreover, $p_\mu$ remains inside $V$ for all $\mu$.
\end{itemize}

We begin by recalling the notion of a Hopf bifurcation, see also \cite{MM76}. Let $\varphi_\mu \colon U \to \mathbb{R}^2$ be a $C^1$ one-parameter family of embeddings of an open set $U \subset \mathbb{R}^2$. Suppose there exists a curve $p_\mu$ of fixed (or periodic) points and a parameter value $\mu_*$ such that:

\begin{itemize}
\item[(a)] $D\varphi_\mu(p_\mu)$ has non-real complex eigenvalues $\lambda(\mu) \neq \overline{\lambda(\mu)}$ for all $\mu$ near $\mu_*$;
\item[(b)] $|\lambda(\mu_*)| = 1$, and $\lambda(\mu_*)^k \neq 1$ for $k = 1,2,3,4$;
\item[(c)] $\dis\frac{d}{d\mu}|\lambda(\mu)| \neq 0$ at $\mu = \mu_*$.
\end{itemize}

Then, as $\mu$ crosses $\mu_*$, the fixed point $p_\mu$ changes its stability: it transitions from a repeller to an attractor or vice versa, depending on the sign of the derivative in condition (c).

If the family $\varphi_\mu$ is $C^5$, then near the bifurcation, it admits the normal form (in polar coordinates)
\begin{equation}\label{eq.normalform}
\begin{aligned}
\varphi_{\mu}(\rho,\theta)
&= F_\mu(\rho,\theta) + O(\rho^5), \\
F_\mu(\rho,\theta)
&= \big(a(\mu)\rho + b_1(\mu)\rho^3,\ \theta + \phi(\mu) + b_2(\mu)\rho^2 \big),
\end{aligned}
\end{equation}
where $a(\mu)$ and $\phi(\mu)$ denote, respectively, the modulus and the argument of $\lambda(\mu)$. Assume further that

\begin{itemize}
\item[(d)] $b_1(\mu_*) \neq 0$.
\end{itemize}

Under these conditions, we say that the family $(\varphi_\mu)_\mu$ \emph{unfolds a ge\-ner\-ic Hopf bifurcation} at $\mu_*$. A characteristic feature of this bifurcation is the emergence of an invariant circle $\mathcal{C}(\varphi_\mu)$, close to the invariant circle of the normal form
\begin{equation}\label{eq.normalformcircle}
\mathcal{C}(F_\mu) = \{(\rho,\theta) : b_1(\mu)\rho^2 = 1 - a(\mu)\}.
\end{equation}
This invariant circle exists either for $\mu > \mu_*$ or for $\mu < \mu_*$, depending on the signs of $da/d\mu$ and $b_1(\mu)$ at $\mu = \mu_*$.

More generally, let $(\psi_\mu)_\mu$ be a family of diffeomorphisms on a manifold, admitting a smooth curve of fixed (or periodic) points $p_\mu$, and suppose there exists $\mu_*$ such that:

\begin{itemize}
\item[(i)] $D\psi_{\mu_*}(p_{\mu_*})$ has a pair of complex conjugate eigenvalues of modulus $1$, while all other eigenvalues lie either strictly inside or strictly outside the unit circle;
\item[(ii)] there exists a smooth family $(N_\mu)_\mu$, defined for all $\mu$ close to $\mu_*$\,, of $\psi_\mu$-invariant (central) manifolds through $p_\mu$\,, of dimension $2$ and tangent to the
eigenspace associated to the pair of complex eigenvalues in (i);
\item[(iii)] the restriction $\varphi_\mu = \psi_\mu|{N_\mu}$ unfolds a generic Hopf bifurcation at $\mu_*$ in the sense of (a)–(d).
\end{itemize}

Then $(\psi_\mu)_\mu$ is also said to \emph{unfold a generic Hopf bifurcation at $\mu_*$}. In this case, as $\mu$ crosses $\mu_*$, the fixed point changes from an attractor or a repeller to a saddle of codimension two (or vice versa), and an invariant circle $\mathcal{C}(\psi_\mu) = \mathcal{C}(\varphi_\mu)$ is created inside the central manifold $N_\mu$.

\medskip

Unfolding a generic Hopf bifurcation is an open property in the space of $C^5$ one-parameter families (the bifurcation parameter $\mu_*$ may vary with the family); see \cite{MM76}. Therefore, to establish the claim above, it suffices to verify it for the reference family $\hat{g}_\mu$.

\begin{proposition}\label{p.existehopf}
The family $(\hat{g}_\mu)_\mu$ unfolds a generic Hopf bifurcation at $\mu_* = 0$, with fixed point $p_\mu = 0$, $a(\mu) = 1 - \mu$, $b_1(\mu) > 0$, and $b_2(\mu) = 0$.
\end{proposition}

\begin{proof}
See \cite[Proposition~2.5]{HV05}.
\end{proof}

We now analyze the local dynamics of $f_\mu = g_\mu|{\{z=0\}}$. Let $V_1 = \{(\rho,\theta) : \rho^2 \le \delta_1\}$ be a neighborhood of $p_\mu = 0$ in the plane $\{z=0\}$, where $\delta_1 > 0$ is given by condition $(C_4)$.

\begin{proposition}\label{p.V1}
For $\mu$ sufficiently close to zero, the following hold:
\begin{enumerate}
\item $f_\mu$ is uniformly expanding outside $V_1$: $\|Df_\mu^{-1}\|^{-1} \ge \sigma_1$;
\item $b_1(\mu)\rho^2 \le 1/1000$ for all $(\rho,\theta) \in V_1$;
\item $\|Df_\mu^{-1} - DF_\mu^{-1}\| \le b_1(\mu)\rho^2/10$ in $V_1$;
\item $|\det Df_\mu - \det DF_\mu| \le 1/100 b_1(\mu)\rho^2/100$ in $V_1$;
\item the normal form $F_\mu$ satisfies $b_1(\mu) > 2|b_2(\mu)|$.
\end{enumerate}
\end{proposition}

\begin{proof}
See \cite[Proposition~2.6]{HV05}.
\end{proof}

In this setting, the fixed point $p_\mu$ transitions from a saddle (with two expanding directions) to an attractor as $\mu$ increases past $\mu_* = 0$. Since $b_1(\mu) > 0$ and $da/d\mu < 0$, the invariant circle $\mathcal{C}(g_\mu)$ is defined for $\mu > \mu_*$. From \eqref{eq.normalformcircle}, the invariant circle of the normal form has radius given by
\[
\rho_0 = \left(\frac{\mu}{b_1(\mu)}\right)^{1/2}.
\]

\noindent More generally, the invariant circle $\mathcal{C}(g_\mu)$ satisfies
\[
a(\mu)\rho + b_1(\mu)\rho^3 + O(\rho^5) = \rho.
\]
For $g_\mu$ close to $\hat{g}_\mu$ and $\mu$ near $\mu_*$, the circle lies close to the origin, and the higher-order terms are negligible. Thus, $\mathcal{C}(g_\mu)$ is contained in an annulus bounded by radii $\rho_1 < \rho_2$ with
\begin{equation}\label{eq.radii}
\rho_i \approx \left(\frac{\mu - \mu_*}{b_1(\mu)}\right)^{1/2}, \quad i = 1,2,
\end{equation}
where $\approx$ denotes equality up to a factor arbitrarily close to $1$.

\begin{remark}\label{r.reparametrization}
After reparametrizing the family, one may assume $\mu_* = 0$ and $a(\mu) = 1 - \mu$. We adopt this normalization throughout.
\end{remark}

\subsection{Markov partitions}
\label{ss.Markov}

For any family of diffeomorphisms $(g_\mu)_\mu$ sufficiently close to $(\hat{g}_\mu)_\mu$, there exists a family $(G_\mu)_\mu$ of Anosov diffeomorphisms of $\mathbb{T}^3$, close to $G$, such that $g_\mu = G_\mu$ outside $V$ for all $\mu$.
For further details, see \cite[Section 2.3]{HV05}.
This ensures that the maps $(g_\mu)_\mu$ admit Markov partitions, although they are not necessarily generating.

The primary use of this property is to ensure that the maps \(g_{\mu }\) admit Markov partitions (which are not necessarily generating):

\begin{remark} 
There exists a Markov partition $\SSS_\mu$ for $g_{\mu }$ such that $V$ is contained in some of Markov rectangles. 
Furthermore, the image of any rectangle intersects at most \(\eta \le 1000\sigma^2\) others.
\end{remark}

To clarify this property, we recall some standard definitions (see also \cite{Bo75}). 
A \emph{Markov rectangle} for an Anosov diffeomorphism \(f \colon M\to M\) is a small compact domain \(S_i \subset M\) equal to the closure of its interior, such that for any \(x, y \in S_i\), the unique intersection point $[x,y]$ in $W^u_{loc}(x) \cap W^s_{loc}(y)$ also lies in \(S_{i}\). Let \(W^s_i(x)\) and \(W^u_i(x)\) denote the connected components of \(W^s(x) \cap S_i\) and \(W^u(x) \cap S_i\) containing \(x\), respectively.

A \emph{Markov partition} is then a finite collection $\SSS$ of Markov rectangles covering \(M\) with pairwise disjoint interiors, satisfying the following invariance condition:
\begin{equation}
\label{eq.markov} 
f(W_i^s(x)) \subset W_j^s(f(x)) \quad\text{and}\quad f(W_i^u(x)) \supset W_j^u(f(x)) 
\end{equation}
for all \(x \in S_i \cap f^{-1}(S_j)\).
Anosov diffeomorphisms and hyperbolic basic sets always admit Markov partitions of arbitrarily small diameter \cite{Bo75, Si68}. 
Such partitions are \emph{generating}, meaning there is at most one point with any given itinerary relative to the partition.

The stable boundary \(\partial ^{s}S_{i}\) consists of points \(x\) that are not in the interior of \(W^u_i(x)\) relative to the local unstable manifold; the unstable boundary \(\partial ^{u}S_{i}\) is defined dually. Continuity of the bracket map \([x,y]\) implies that \(\partial ^{s}S_{i}\) (resp. \(\partial ^{u}S_{i}\)) is a union of stable (resp. unstable) sets $W_i^s(z)$ (resp. $W_i^u(z))$. 
We denote the collective boundaries by $\partial^s \SSS = \cup_i \partial^s S_i$ and $\partial^u \SSS = \cup_i \partial^u S_i$. 
By \eqref{eq.markov}, these satisfy $f(\partial^s \SSS) \subset \partial^s \SSS$ and $f^{-1}(\partial^u \SSS) \subset \partial^u \SSS$.

Finally, a Markov partition allows us to define a quotient map \(\phi \) on the space of stable leaves \(\{W^s_i(x)\}\). 
This map sends each \(W_i^s(x)\) to the stable leaf containing its image ($\phi$ may be multivalued at the boundaries of the Markov rectangles). 
The map \(\phi \) is uniformly expanding and satisfies the Markov property: if \(R_i = \{W_i^s(x) : x \in S_i\}\), then
\begin{equation}
\label{eq.markov2}  
\phi(\text{int}(R_i)) \cap \text{int}(R_j) \neq \emptyset \implies \phi(R_i) \supset R_j.  
\end{equation}

Returning to our setting, recall that we have fixed a Markov partition $\SSS$ for the Anosov diffeomorphism $G$, such that the closure of $V$ is contained in the interior of some Markov rectangle $S_0 \in \SSS$.
By construction, each $G_\mu$ is $C^1$-close to $G$. 
Since Anosov diffeomorphisms are structurally stable, each $G_\mu$ is topologically conjugate to $G$: there exists a homeomorphism $h_\mu:M\to M$ such that 
$$
G_\mu\circ h_\mu = h_\mu\circ G.
$$
Using this conjugacy, we define a family of partitions
$$
\SSS_\mu = \big\{S_{i,\mu}=h_\mu(S_i)\ :\ S_i \in \SSS \big\},
$$
which yields a Markov partition for $G_\mu$. 
Moreover, the conjugacy $h_\mu$ is $C^0$-close to the identity whenever $G_\mu$ is close to $G$.

If the family $(g_\mu)_\mu$ is sufficiently close to $(\hat{g}_\mu)_\mu$, we can ensure that the closure of $V$ is contained in the interior of $S_{0,\mu}=h_\mu(S_0)$ for every $\mu$. 
Since $g_\mu$ and $G_\mu$ coincide outside $V$, it follows that $g_\mu(S_{i,\mu})=G_\mu(S_{i,\mu})$ for every $i$, including $i=0$.
Furthermore, since the forward iterates of all points in $\partial^s \SSS_\mu$ under $G_\mu$ remain outside the region $V$, they coincide with the corresponding iterates under $g_\mu$.

The same is true for the partially hyperbolic diffeomorphisms $g_\mu$, with strong-stable leaves playing the role of stable manifolds. 
Indeed, the forward iterates of the points in $\partial^s (h_\mu(S_i))$ never enter the perturbation region $V$, and thus their local strong-stable leaves for $g_\mu$ coincide with their local stable manifolds for $G_\mu$.
As a consequence, we can still define a quotient map on the space of local strong-stable leaves of $g_\mu$, which retains the Markov property in the sense of \eqref{eq.markov2}: the domains $R_{i,\mu}$ for the quotient maps of $g_\mu$ and $G_\mu$ coincide, as do their images under these maps.

It is in this context that we refer to \(\mathcal{S}_{\mu }\) as a Markov partition for the partially hyperbolic maps \(g_{\mu }\), though we remark that they are generally not generating. 

The control over the number of intersecting rectangles stems from the construction of \(\mathcal{S}_{\mu }\) relative to the initial map \(G\). 
Heuristically, the number of intersections is governed by the unstable/central Jacobian, \(\sigma ^{2}\); the factor of \(1000\) provides a safe margin for oscillations arising from variations in the diameters of the rectangles.

\subsection{Summary}
\label{s.summary}

We summarize the main conclusions of this section. The proofs of our results rely only on the properties $(H_1)$–$(H_6)$ listed below.

We consider one-parameter families $(g_\mu)_\mu$, $\mu \in [-1,1]$, of $C^r$ diffeomorphisms, $r \geq 5$, on the three-dimensional torus $M = \mathbb{T}^3$, satisfying the following conditions:

\begin{itemize}
\item[$(H_1)$] There exist an open set $V \subset M$ and a continuous family $(G_\mu)_\mu$ of transitive Anosov diffeomorphisms such that
\[
\|DG_\mu^{-1} \mid E^u_\mu\|^{-1} \ge \sigma > 3,
\]
with $g_{-1} = G_{-1}$ and $g_\mu = G_\mu$ outside $V$ for all $\mu$. Moreover, for each $\mu$, the set $V$ is contained in the interior of some  rectangle of a Markov partition $\mathcal{S}_\mu$ for $G_\mu$.

\item[$(H_2)$] There exists a curve $(p_\mu)_\mu$ of fixed (or periodic) points and a parameter $\mu_*$ such that $p_\mu$ is a hyperbolic saddle for $g_\mu$ when $\mu < \mu_*$, undergoes a generic Hopf bifurcation at $\mu = \mu_*$, and becomes an attractor for $\mu > \mu_*$. Moreover, $p_\mu$ remains inside $V$ for all $\mu$.

\item[$(H_3)$] For all $\mu$ sufficiently close to zero, the diffeomorphisms $g_\mu$ are partially hyperbolic: there exists a $Dg_\mu$-invariant dominated splitting
\[
TM = E_\mu^{ss} \oplus E_\mu^{cu},
\]
where $E_\mu^{ss}$ is uniformly contracting.

\item[$(H_4)$] There exists $n \ge 1$ such that, for every $\mu$ sufficiently close to zero,
\[
\sup_M \big\|(D g_\mu^n \mid E^{cu}_\mu)^{-1}\big\| \cdot
\big\|D g_\mu^n \mid E^{ss}_\mu \big\| \cdot
\big\|D g_\mu^n \mid E^{cu}_\mu \big\| < 1.
\]

\item[$(H_5)$] There exists a neighborhood $V_1$ of $p_\mu$ inside the center-unstable manifold $W^{cu}_\mu(p_\mu)$, with fixed radius, satisfying (1)–(5) of Proposition~\ref{p.V1}.

\item[$(H_6)$] The Markov partition $\mathcal{S}_\mu$ for $g_\mu$ can be chosen so that the image of any rectangle intersects at most $\eta \le 1000\sigma^2$ rectangles.
\end{itemize}

\section{Proof of Theorems~\ref{t.thmB} and~\ref{t.thmC}}
\label{s.proofs}

Theorems~\ref{t.thmB} and~\ref{t.thmC} are proved using Theorem~\ref{t.thmA}, the proof of which is provided in Section~\ref{s.mwh}.
Accordingly, we start by constructing a map with holes satisfying conditions $(\A_1)$ and $(\A_2)$.

\subsection{Constructing a map with holes}
\label{s.construction}

Our first step is to associate to each diffeomorphism $g_\mu$ a two-dimensional map with holes $f_\mu$. Let $\mathcal{S}_\mu = \{S_0, S_1, \ldots, S_m\}$ be a Markov partition for $g_\mu$ as in Section~\ref{ss.Markov}. 
For each $x \in S_i$, we denote by $W_i^s(x)$ the connected component of $W^{ss}(x) \cap S_i$ containing $x$.

For each $i \ge 1$, fix a domain $R_{i,\mu} \subset W^c(p_\mu)$ that intersects each stable leaf $W_i^s(x)$ in exactly one point. For $i=0$, we proceed similarly, but denote the domain by $R_{0,\mu}^*$ and require that $p_\mu \in R_{0,\mu}^*$. Define
\[
W_\mu = R_{0,\mu}^* \cup R_{1,\mu} \cup \cdots \cup R_{m,\mu}.
\]
Note that $W_\mu$ is not necessarily connected. Let
\begin{equation}
\label{eq.Hmu}
H_\mu = R_{0,\mu}^* \cap W_{\mathrm{loc}}^s(p_\mu)
\quad \text{and} \quad
R_{0,\mu} = R_{0,\mu}^* \setminus H_\mu.
\end{equation}

Define a projection $\pi_\mu \colon M \to W_\mu$ along the strong-stable leaves $W_j^s(x)$ within each rectangle $S_j$. By Corollary~\ref{c.folheacao}, the restriction of $\pi_\mu$ to each rectangle is of class $C^{1+\delta}$. We then define
\[
f_\mu = \pi_\mu \circ g_\mu \colon W_\mu \to W_\mu.
\]
Since $\pi_\mu$ is multivalued on the boundaries of the rectangles, the map $f_\mu$ is a multivalued $C^{1+\delta}$ map. We will show that $f_\mu$ fits into the framework of maps with holes introduced in $(\mathcal{A}_1)$, $(\mathcal{A}_2)$ and Remark~\ref{r.weakmarkov}.

\subsection{Verification of the Markov and hyperbolicity conditions}

We first verify the Markov condition $(\mathcal{A}_1)$ for $\mu > 0$ sufficiently small. For simplicity, we often suppress the parameter $\mu$ from the notation of the Markov rectangles. Define $H_0 = H_\mu$ and $H_i = \emptyset$ for $i \ge 1$. By construction, $H_0$ has nonempty interior.

Since $\mathcal{S}_\mu$ is a Markov partition for $g_\mu$, each image $f_\mu(R_i)$ is a union of domains $R_j$, for $j$ in some index set $J(i)$. Moreover, the Anosov diffeomorphisms $G_\mu$ are transitive, since they are conjugate to a linear Anosov map. It follows that, for each $0 \le i \le m$, there exists $\ell \ge 0$ such that $f_\mu^\ell(R_i)$ contains $R_0$. In addition,
$H_0 = f_\mu(R_0) \setminus \bigcup_{j=0}^m R_j$
has nonempty interior. Thus, the Markov property holds in the sense of Remark~\ref{r.weakmarkov}.

To complete the verification of $(\mathcal{A}_1)$, it remains to check that the domains $R_i$ have bounded inner diameter and that the box dimension of their boundaries is strictly smaller than $\dim W_\mu$. These properties follow from the next two results, stated in a slightly more general setting for Markov partitions of Anosov diffeomorphisms.

\begin{proposition}
\label{p.inner}
Let $\Phi \colon M \to M$ be an Anosov diffeomorphism with stable index $1$, and let $\mathcal{R}$ be a generating Markov partition. Then there exists $L > 0$ such that, for any $x,y$ in the same Markov rectangle $R \in \mathcal{R}$, there exists a piecewise smooth curve $\gamma \subset R$ connecting $x$ to $y$ with $\length(\gamma) \le L$.
\end{proposition}

\begin{proof}
See \cite[Proposition~3.4]{HV05}.
\end{proof}

\begin{proposition}
\label{p.limit_capacity}
Let $\Phi \colon M \to M$ be an Anosov diffeomorphism and $\mathcal{R}$ a Markov partition. Then the box dimension of the unstable boundary $\partial^u \mathcal{R}$ is strictly smaller than the unstable dimension of $\Phi$.
\end{proposition}

\begin{proof}
See \cite[Proposition~3.5]{HV05}.
\end{proof}

\medskip

We now turn to the hyperbolicity condition $(\mathcal{A}_2)$. Define
\[
\zeta(n,\mu_{f_\mu}) = \frac{\mu}{-4\log \mu} n^2 e^{-\frac{1}{4}\mu n}.
\]
The required estimate on the volume of bad cylinders is given by the following result.

\begin{proposition}
\label{cresc_exp}
For the family $(f_\mu)_\mu$ constructed above, there exist constants $n_0, c_0, \mu_0 > 0$ such that
\[
\Leb_2\big(B_n(c_0 \mu_{f_\mu})\big)
\le \zeta(n,\mu_{f_\mu})
\]
for all $n \ge n_0$ and $0 \le \mu < \mu_0$, where $\Leb_2$ denotes the normalized Lebesgue measure induced by the Riemannian metric.
\end{proposition}

The proof of this proposition is technical and is deferred to Section~\ref{s.lyap}.

\subsection{Proof of Theorems~\ref{t.thmB} and~\ref{t.thmC}}
\label{s.proof_thmB}

In the previous section, we verified that $f_\mu$ satisfies conditions $(\A_1)$ and $(\A_2)$. Since $f_\mu$ is obtained by restricting the diffeomorphisms $g_\mu$ to compact domains the norms $\|Df_\mu^{-1}\|$, $\mu \in [-1,1]$, are uniformly bounded.

Given $\varepsilon > 0$, Theorem~\ref{t.thmA} yields $\delta > 0$ such that, whenever $\Leb_2(H_\mu) < \delta$, one has $\HD(\Lambda_{f_\mu}) > 2 - \varepsilon$.
Combining \eqref{eq.radii} with the definition of $H_\mu$ in \eqref{eq.Hmu}, the volume $\Leb_2(H_\mu)$  is equal to $\mu$ up to a multiplicative constant.
In other words, given $\varepsilon > 0$  there is $\mu_0 > 0$ such that for all $\mu \le \mu_0$, $\HD(\Lambda_{f_\mu}) > 2 - \varepsilon$.
This proves Theorem~\ref{t.thmC}.

Note that 
$$
\Lambda_\mu = \bigcup_{x \in \Lambda_{f_\mu}} \big( \{x\} \times (\mathcal{F}^{ss}_\mu(x)) \big),
$$
Then, Theorem~\ref{t.thmB} follows from the fact that the holonomy along leaves of $\mathcal{F}^{ss}_\mu$ defined in Section \ref{ss.partialhyperbolicity} are of class $C^{1+\beta}$ (see Corollary~\ref{c.folheacao})
consequently Lipschitz, implies that
$$
\HD(\Lambda_\mu ) \ge \HD(\Lambda_{f_\mu}) + 1 > 3 - \varepsilon.
$$
Since $\varepsilon > 0$ is arbitrary, this completes the proof of Theorem~\ref{t.thmB}.

\section{Proof of Theorem~\ref{t.thmA}}
\label{s.mwh}

Let $f$ be a map with holes satisfying conditions $(\A_1)$ and $(\A_2)$.
For a fixed positive constant $c_0$ and some $0 \le \lf < \mu_0$, we define $S_0(c_0\lf)$ as the set of points that belong to some cylinder $C(\alpha_1)$ such that $ \phi_1(\alpha_1) > c_0\lf$.
Recall that the average least expansion $\phi_j$ of the $j$-cylinder $C(\alpha_1, \cdots , \alpha_j)$, $\alpha_i \in \{0,1,\cdots m\}$, is defined in \eqref{eq.phi}.

For  $n \ge 1$, let $f^{(-n)}$ denote an inverse branch of $f^n$ and define
\begin{align}
	\label{eq12}
	S_n(c_0\lf) & = B_n(c_0\lf)\setminus \left(B_{n + 1}(c_0\lf) \cup
	f^{(-n)}(H_f) \right) \nonumber \\
	& = \Big\{x \in C(\alpha_1,...,\alpha_n,\alpha_{n+1}) \colon
	\phi_{n+1}(\alpha_1,...,\alpha_n,\alpha_{n+1})> c_0\lf \; \\
	& \qquad \text{ and }\; \phi_j(\alpha_1,...,\alpha_j)\leq c_0\lf \;{\rm
		\text{ for all }}\;j \in\{1,...,n\}\Big\}, \nonumber
\end{align}
where $H_f$ is as in \eqref{eq.hf}. 
Observe that the family $\left\{ S_n(c_0\lf)\right\}_{n \ge 0}$ consists of pairwise disjoint sets. 
Moreover, for $n \geq 0$, the sets $S_n(c_0\lf)$ can be expressed as a union of $(n + 1)$-cylinders.
In particular, $f^{n+1}$ is well defined on $S_n(c_0\lf)$. 

Given $c > 0$, $F$ is called {\emph{$c$-expanding}} if there exists $N \ge 1$ such that every inverse branch  $F^{(-N)}$ satisfies $\|DF^{(-N)}(x) \| \leq e^{-cN}$ for all $x$ in its domain
	
For a positive integer $n$, consider the map induced by $f$:  $F_n : \bigcup\limits_{k = 0}^{n-1} S_k(c_0\lf)
	\rightarrow M$, defined by
	\begin{equation}\label{def4}
		F_n(x) = f^{k+1}(x)\quad \mbox{ if }\quad x \in S_k(c_0\lf), \quad k=0, \dots , n-1.
	\end{equation}
	
	We have the following proposition.	
	
	\begin{proposition}
		\label{p.F-expanding}
		$F_n$ is a $(c_0\lf)$-expanding map with holes.
	\end{proposition}
	
	\begin{proof}
		Let $x$ be in $\bigcup\limits_{k = 0}^{n-1} S_k(c_0\lf)$, and suppose $x \in S_{j-1}(c_0\lf)$, with $j\in \{1,\dots, n\}$.
		Thus, there are $\alpha_1,\dots ,\alpha_{j}$ such that
		$x \in C(\alpha_1,\dots ,\alpha_{j})$ with
		$$
		\phi_j(\alpha_1,\dots,\alpha_{j}) > c_0\lf.
		$$
		Hence,
		$$
		\log\prod\limits_{i = 1}^{j}\|Df^{-1}(f^i(x))\|^{-1} =
		\sum\limits_{i = 1}^{j}\log\|Df^{-1}(f^i(x))\|^{-1}  > (c_0\lf)j
		$$
		Therefore,
		$$
		\prod\limits_{i = 1}^{j}\|Df^{-1}(f^i(x))\| < e^{-(c_0\lf)j}.
		$$
		On the other hand, the fact that $x \in S_{j-1}(c_0\lf)$ implies
		that $F_n(x) = f^{j}(x)$.
		Then
		$$
		\|DF_n^{(-1)}(F_n(x))\| = \|Df^{-j}(f^j(x))\| \leq
		\prod \limits_{i = 1}^{j}\|Df^{-1}(f^i(x))\| < e^{-(c_0\lf)j}.
		$$
		Therefore $F_n$ is $(c_0\lf)$-expanding.
	\end{proof}
	
	We define the \emph{repeller} of $ F_n $  by
	\begin{equation}
		\Lambda_{F_n}  =  \bigg\{x \in M \colon F_n^j(x) \in \bigcup\limits_{k = 0}^{n-1}S_k(c_0\lf), \forall j
		\geq 0 \bigg\} \;\;\;\;\;
	\end{equation}
	and the \emph{hole} of $F_n$ by
	\begin{equation}
		\label{eq_holeF}
		H_{F_n} = M \backslash \bigcup\limits_{k = 0}^{n-1}S_k(c_0\lf).
	\end{equation}
	Notice that $\Lambda_{F_n} \subset \Lambda_f$ and, in general, $H_{F_n}
	\nsubseteq H_f$.
	
	\begin{proposition}
		\label{Teo5}
		Given $c_0 > 0 $, let $F$ be a $(c_0\mu_f)$-expanding map with holes $H_F$.
		There exists a map
		$$
		\psi : [0,1] \rightarrow [0,d],
		$$
		where $\psi(x)$ tends to $d = dim M$ when $x$ converges to zero,
		such that
		$$
		HD(\Lambda_F) \geq \psi(Leb (H_F)).
		$$
	\end{proposition}
	
	\begin{proof}
		See \cite[Theorem 3]{Dys05}.
	\end{proof}
	
	\medskip
	
	In other words, given $\epsilon > 0$, there is $\delta = \delta(n)
	> 0$ such that if $\Leb(H_{F}) < \delta$ then $\HD(\Lambda_{F}) \geq d -
	\epsilon$.
	
	\medskip
	
	To prove Theorem~\ref{t.thmA}, we consider $F_n$ as above.
	Notice that $B_n(c_0\lf)$ is not contained in the domain of $F_n$, so the hole of $F_n$ can be written as
	\begin{align*}
		H_{F_n} = & B_n(c_0\lf) \cup H_f \cup (f^{-1}(H_f)\cap B_1(c_0\lf)) \cup \dots \cup \\
		& \cup \dots \cup (f^{-(n - 1)}(H_f) \cap B_{n -
			1}(c_0\lf)),
	\end{align*}
	where $H_f$ is the hole of $f$ (see \eqref{eq.hf}). 
	
	Then
	\begin{equation}
		\label{eq.vol1}
		\Leb(H_{F_n}) \le \Leb(B_n(c_0\lf)) + \Leb(H_f) + \sum\limits_{j = 1}^{n - 1}\Leb(f^{-j}(H_f)\cap
		B_j(c_0\lf)).
	\end{equation}

	By hypothesis $\sup \|Df^{(-1)}\| \le S$.
	Hence
	$$
	\Leb (f^{-j} (H_f)) \le S^{dj} \Leb(H_f).
	$$
	Moreover, the number of connected components of $B_j(c_0 \mu_f) \le (m+1)^j$.
	Therefore,
	\begin{equation}
		\label{eq.vol2}
		\Leb(f^{-j}(H_f) \cap B_j(c_0\lf)) \le S^{dj} \Leb(H_f) (m+1)^j.
	\end{equation}

	By condition $(\A_2)$, $\Leb(B_n(c_0\lf)) \leq \zeta(n, \mu_f)$, for $n \ge n_0$ and $0 \le \lf < \mu_0$.
	Therefore, from Equations~\eqref{eq.vol1} and \eqref{eq.vol2}, we have
	\begin{equation}\label{eq24}
		\Leb(H_{F_{n}}) \leq \zeta(n, \mu_f) + \lf + \lf \sum\limits_{j = 1}^{n - 1}S^{dj}(m + 1)^j.
	\end{equation}
By reducing $\mu_0 > 0$ if necessary, we have for every $0 \le \lf \le \mu_0$, 
	$$
	\zeta(n_0, \mu_f) + \lf\sum\limits_{j = 0}^{n_0 - 1}S^{dj}(m + 1)^j  < \delta,
	$$
	By \eqref{eq24}, we have $\Leb(H_{F_{n_0}}) < \delta$.
	From Proposition~\ref{p.F-expanding}, $F_{n_0}$ is $(c_0\lf)$-expanding.
	Therefore, it follows from Proposition~\ref{Teo5} that $\HD(\Lambda_{F_{n_0}}) \geq d -
	\epsilon$.
	As $\Lambda_{F_{n_0}} \subset \Lambda_f$, it follows that $\HD(\Lambda_f) \geq d - \epsilon$.
	The proof of Theorem~\ref{t.thmA} is complete.
	
	\section{Central Lyapunov Exponents}
	\label{s.lyap}
	
	We are left to prove Proposition~\ref{cresc_exp}. The proof, 
	given in Section~\ref{proof}, combines estimates of the derivative
	and the Jacobian of $f_\mu$ near the bifurcation point $p_\mu$
	(Section~\ref{local}) with a combinatorial analysis of the visits
	of orbits to a neighborhood of $p_\mu$ (Section~\ref{visits}).
	
	The intuition is as follows. For $\mu > 0$, inside
	the Markov rectangle $S_0\supset V$, the derivative of $g_\mu$ may
	contract the central bundle in certain directions, in contrast to what
	happens outside $S_0$, where the derivative  uniformly expands the central bundle. In principle, an orbit may spend a significant amount of time in $S_0$ accumulating contraction along the central
	bundle, and the time spent outside $S_0$ may not be sufficient to
	compensate for these contractions. However, for a set of full Lebesgue measure and for every sufficiently small $\mu>0$, expansion does prevail, as
	required by ($\A_2$).
	
	Our estimates involve a large number of constants. Rather than keeping track of all relations between these constants, we choose to use explicit values whenever possible. However, these values are mainly technical and are not intended to be optimal.

	\subsection{Iterates in the non-hyperbolic region}
	\label{local}
	
	Lemma~\ref{l.metric} states that $f_\mu$ is not excessively contracting along
	the central direction outside a small neighborhood of $p_\mu$.
	Lemma~\ref{l.jac_local} states that $f_\mu$ is volume expanding
	outside $H_\mu$.
	For simplicity, we continue to denote $W_\mu \cap V$ by $V$.
	Recall that $V_1\subset V$ is the neighborhood of $p_\mu$
	introduced in property ($H_5$) and Proposition~\ref{p.V1}.
	
	Let $V_2(\mu)$ be a neighborhood of $p_\mu$ inside $H_\mu$, defined by
	
	\begin{equation}
		\left\{(\rho,\theta)\;:\; \rho < \left(\dis\frac{127}{128}\right)^2\rho_0(\mu)\ \right\},
	\end{equation}
	
	\noindent where $\rho_0(\mu)=(\mu/b_1(\mu))^{1/2}$ is the
	radius of the $F_\mu$-invariant circle created by the Hopf bifurcation.
	
	\begin{figure}[phtb]
		\begin{center}
			\includegraphics[height=7.0cm]{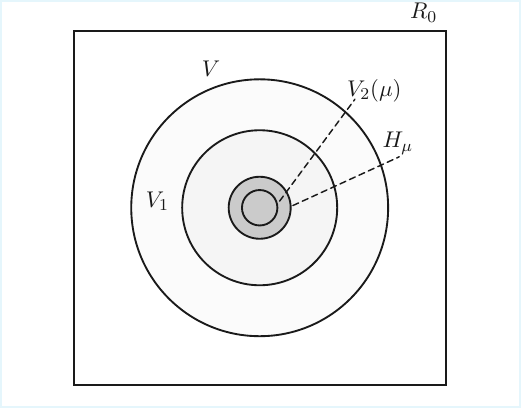}
			\caption{Neighborhood of $p_\mu$}
		\end{center}
	\end{figure}
	
	\begin{lemma}
		\label{l.metric}
		Let $(g_\mu)_\mu$ satisfy properties ($H_1$)--($H_6$).
		Then there exists a constant $\mu_1>0$ such that, for every
		$0 < \mu \le \mu_1$, we have $H_\mu\subset V_1$, and there exists
		a neighborhood $V_2(\mu)\subset H_\mu$ of $p_\mu$ such that:
		\begin{itemize}
			\item[(a)] $\displaystyle \log \big\| (Df_\mu(x))^{-1} \big\|^{-1}
			\ge -\frac{33}{32}\mu + \frac{31}{32}b_1(\mu)\rho^2
			\text{ for every }x  \in  V_1$\,.
			\item[(b)]
			$\displaystyle \log \big\|(Df_\mu(x))^{-1}\big\|^{-1}
			\ge - \frac{3}{32}\mu$ for every $x$ outside $V_2(\mu)$.
		\end{itemize}
	\end{lemma}
	
	\begin{proof}
		
	See \cite[Lemma~4.1]{HV05}.
\end{proof}
	
	\bigskip
	
	We now estimate the Jacobian $\Jac f_\mu(x) = |\det Df_\mu (x)|$ of $f_\mu$.
	
	\begin{lemma}\label{l.jac_local}
		For every $0 < \mu \leq \mu_1$, we have
		$\displaystyle \log \Jac f_\mu(x) \ge 2\log \sigma_1>0$ for every $x$ outside $V_1$. Moreover:
		\begin{itemize}
			\item[(a)] $\displaystyle \log \Jac f_\mu(x) \geq -\frac{65}{32}\mu+\frac{127}{32}
			b_1(\mu)\rho^2$ for every $x \in V_1$
			\item[(b)] $\displaystyle \log \Jac f_\mu(x) \geq \frac{61}{32}\mu$ for every
			$x$ outside $V_2(\mu)$.
		\end{itemize}
	\end{lemma}
	
	\begin{proof}
	See \cite[Lemma~4.2]{HV05}.
	\end{proof}

	\subsection{Visits to the non-hyperbolic region}
	\label{visits}
	
	We deduce more global versions of the estimates from
	the previous section that apply to orbit segments visiting
	the non-hyperbolic region several times.
	
	Given integers $n\ge 1$, $t\ge 1$ and sequences $k_j$, $l_j$
	for $1\le j\le t$, such that
	$$
	k_1+l_1+\cdots+k_t+l_t=n,
	$$
	we denote by
	$\CC(k_1,l_1,\dots , k_t,l_t)$ the set of all 
	$n$-cylinders $C_n$ that spend, alternately, $k_i$ iterates outside
	$R_0$ and then $l_i$ iterates inside $R_0$\,. More precisely, a cylinder 
	$C_n=C(\alpha_0, \ldots, \alpha_{n-1})$ belongs to $\CC(k_1,l_1,\dots
	, k_t, l_t)$ if and only if $\alpha_j>0$ whenever $j$ lies in the
	time intervals
	$$
	[0,k_1),[k_1+l_1\,,k_1+l_1+k_2),\ldots,[k_1+\cdots+l_{t-1}\,,k_1+\cdots+l_{t-1}+k_t)
	$$
	and $\alpha_j=0$ for all other values of $j$. Here $k_1\ge 0$
	and $l_t\ge 0$, while all other $k_i$ and $l_i$ are strictly
	positive. Given $0\le l \le n$, we denote by $\CC(n,l,t)$ the union
	of all $\CC(k_1,l_1,\dots , k_t, l_t)$ such that $l_1+\cdots+l_t=l$.
	
	\smallskip
	
	Recall that $\eta\ge 1$ is an upper bound for the number of Markov rectangles that the image of any rectangle can intersect, see $(H_6)$.
	
	\begin{lemma}
		\label{l.combinatorial}
		For every $n, l, t, k_1\,, l_1\,,\ldots, k_t\,, l_t$
		\begin{enumerate}
			\item $\# \CC(k_1\,,l_1\,,\ldots,k_t\,,l_t) \le \eta^k$ where
			$k=k_1+\cdots+k_t=n-l$
			\item $\displaystyle \# \CC(n,l,t) \le \binom{l}{t-1}\binom{n-l}{t-1} \eta^k$
		\end{enumerate}
	\end{lemma}
	
	\begin{proof}
	See \cite[Lemma~4.3]{HV05}.
	\end{proof}
	
	\begin{lemma}\label{l.medida}
		For every $k_1\,,l_1\,,\ldots, k_t\,,l_t$ and
		$C_n\in\CC(k_1\,,l_1\,,\ldots, k_t\,,l_t)$ there exist points $x_i
		\in f_\mu^{k_1+l_1+\cdots+k_i}(C_n)\subset R_0'$
		so that
		$$
		\Leb_2(C_n)
		\le \sigma^{-2(n-l)} \prod_{i=1}^t \Jac f_\mu^{l_i}(x_i)^{-1},
		$$
		where $\sigma>3$ is such that $\|Df_\mu^{-1}(x)\|^{-1}\ge \sigma$ for every $x$ outside $V$.
	\end{lemma}
	
	\begin{proof}
	See \cite[Lemma~4.4]{HV05}.
\end{proof}

The next two lemmas provide estimates for the derivative and the Jacobian along the central direction along an orbit.
	
	\begin{lemma}
		\label{l.derivada}
		There is a positive integer $n_0$ such that for every $n > n_0$, every choice of  $k_1\,,l_1\,,\ldots, k_t\,,l_t$, $\mu$ sufficiently small, and every cylinder 
		$C_n\in\CC(k_1\,,l_1\,,\ldots, k_t\,,l_t)$, we have
		\[
		\sum^n_{j=1} \inf_{x\in C_n} \log
		\big\|(Df_\mu(f_\mu^j(x)))^{-1}\big\|^{-1} \ge -
		\frac{3}{32}l\mu - \frac{13}{32}t\log \mu + (n-l) \log \sigma .
		\]
	\end{lemma}
	
	\begin{proof} 
This proof is based on the one provided by \cite[Lemma~4.5]{HV05}, the difference is that here we explicitly state the existence of $n_0$.
	
Given neighborhoods $U_1\subset U_2$ of $p_\mu$, we say that an orbit segment $\OO$ \emph{crosses} $U_2\setminus U_1$ if its
	first iterate lies in $U_1$, its last iterate lies outside $U_2$, and all intermediate iterates lie in $U_2 \setminus U_1$. Let $V_\mu$ be the neighborhood
	of $p_\mu$ corresponding to $\rho = 2\rho_0(\mu)$ in local coordinates, as in Lemma~\ref{l.metric}.

	We consider the worst-case scenario, namely orbits that cross $V_1\setminus V_\mu$ every time they visit $R_0$.
	The arguments are valid in general; in fact, the estimates from Lemma~\ref{l.metric} are better in all other cases.
	For simplicity, we assume that each crossing segment spends $q_0$ iterates in $R_0 \setminus V$,	$q_1$ iterates in $V\setminus V_1$, and $q_\mu$ iterates in
	$V_1 \setminus V_\mu$. This simplification is harmless, since these numbers may vary by at most a fixed finite amount. We then take $n_0 = q_0 + q_1 + q_\mu$, which corresponds to the time needed to cross $R_0  \setminus V_\mu$.
	
	For notational simplicity, we write
	$\big\| (Df_\mu(f_\mu^j(x)))^{-1}\big\|^{-1} = \lambda_\mu (f_\mu^j(x))$.
	During each visit to $R_0$, the orbit crosses each
	region $R_0\setminus V$, $V\setminus V_1$ and $V_1\setminus V_\mu$
	exactly once. If $l_i$ denotes the number of iterates spent during a visit to $R_0$, we obtain a decomposition of the sum along these regions.
	\begin{align*}
		\sum^{l_i}_{j=1} & \inf_{x\in C_n}  \log
		\big\| (Df_\mu(f_\mu^j(x)))^{-1}\big\|^{-1}
		= \sum_{j=1}^{l_i} \inf_{x\in C_n}
		\log \lambda_\mu (f_\mu^j (x)) \\
		& \ge \sum_{j=1}^{q_0} \inf_{x\in C_n}
		\log \lambda_\mu (f_\mu^j(x)) +
		\sum_{j=q_0+1}^{q_0+q_1} \inf_{x\in C_n}
		\log \lambda_\mu (f_\mu^j(x)) + \\
		& \quad + \sum_{j=q_0+q_1+1}^{q_0+q_1+q_\mu} \inf_{x\in C_n}
		\log \lambda_\mu (f_\mu^j(x))
		+ \sum_{j=q_0+q_1+q_\mu+1}^{l_i} \inf_{x\in C_n}
		\log \lambda_\mu (f_\mu^j(x)).
	\end{align*}
	
It follow from Lemma~\ref{l.metric} that
	\begin{align*}
		\sum_{j=1}^{l_i} \inf_{x\in C_n}
		\log \lambda_\mu (f_\mu^j (x)) & \ge
		q_0\log \sigma + q_1\sigma_1 + \sum_{j=1}^{q_\mu} \bigg(-\frac{33}{32}\mu
		+ \frac{31}{32} b_1(\mu) \rho_j^2\bigg) - \\
		& \qquad \qquad - \frac{3}{32}\mu (l_i-q_\mu-q_1-q_0) \\
		& \ge \sum_{j=1}^{q_\mu} \bigg(-\frac{33}{32}\mu +
		\frac{31}{32} b_1(\mu) \rho_j^2\bigg) - \frac{3}{32}\mu l_i,
	\end{align*}
	where $\rho_j$ is the $\rho$-coordinate of $f_\mu^j (x)$, $x\in
	\partial V_\mu$, in the local system of coordinates.
	
	Let $\rho_j$ denote the $\rho$-coordinate of $f_\mu^j(2\rho_0)$, $j\ge 1$, then
	$$
	\frac{\rho_{j+1}}{\rho_j} = \frac{(1-\mu)\rho_j+b_1(\mu)\rho^3_j}{\rho_j} =
	(1-\mu)+b_1(\mu)\rho^2_j.
	$$
	
	Let $\hat{\rho} = f_\mu^{q_\mu+1}(2\rho_0)$ be  $\rho$-coordinate of the
	first iterate of $2\rho_0$ outside $V_1$.
	The fact that for $\mu$ close to zero $\rho_0$ is close to zero, implies that
	$$
	\frac{\rho_1}{\hat{\rho}}
	= \frac{(1-\mu)2\rho_0+b_1(\mu)(2\rho_0)^3}{\hat{\rho}}
	\le b_1(\mu)^{30/32}\rho_0^{15/32}= \mu^{15/32},
	$$
	for $\mu$ sufficiently small. Hence,
	$$
	1
	= \frac{\rho_1}{\hat{\rho}} \prod_{j=1}^{q_\mu} \frac{\rho_{j+1}}{\rho_j}
	\le \mu^{15/32}\prod_{j=1}^{q_\mu} \big((1-\mu)+b_1(\mu)\rho^2_j
	\big).
	$$
	Using the elementary inequality
	$$
	\log(1-a+b)\le -\frac{31}{32} a + \frac{33}{32} b
	\quad\text{for } 0 < a, b \le \frac{1}{32},
	$$
	we deduce that, for sufficiently small $\mu>0$,
	\begin{equation}
		\label{e.log_mu}
		\sum_{j=1}^{q_\mu} \bigg( -\frac{31}{32}\mu + \frac{33}{32}
		b_1(\mu)\rho_j^2 \bigg) \geq -\frac{15}{32}\log \mu.
	\end{equation}
	Moreover, since $\rho \geq 2\rho_0$, we have
	$b_1(\mu)\rho^2 \geq 4b_1(\mu)\rho^2_0 = 4\mu$,	and therefore
	$$
	\sum_{j=1}^{q_\mu} \bigg( -\frac{33}{32}\mu
	+ \frac{31}{32}b_1(\mu)\rho_j^2 \bigg) \geq \frac{13}{15}
	\sum_{j=1}^{q_\mu} \bigg( -\frac{31}{32}\mu + \frac{33}{32}
	b_1(\mu)\rho_j^2 \bigg) \geq - \frac{13}{32} \log \mu.
	$$
	Hence, for $n > n_0$,
	\begin{align*}
		\sum^n_{j=1} \inf_{x\in C_n} \log \big\| &
		(Df_\mu(f_\mu^j(x)))^{-1}
		\big\|^{-1} =  \\
		&  = \sum^t_{i=1} \bigg[ \sum_{j=1}^{l_i} \inf_{x\in C_n} \log
		\lambda_\mu (f_\mu^j (x)) + \sum_{j=1}^{k_i}
		\inf_{x\in C_n} \log
		\lambda_\mu (f_\mu^j(x))\bigg] \\
		& \ge \sum^t_{i=1} \bigg(
		\sum_{j=1}^{q_\mu} \bigg(-\frac{33}{32} \mu + \frac{31}{32}
		b_1(\mu)\rho^2_j\bigg)-\frac{3}{32}\mu l_i+k_i \log \sigma \bigg) \\
		& \ge - \frac{13}{32}t\log \mu - \frac{3}{32}\mu l + (n-l)\log \sigma.
	\end{align*}
	We have assumed $\lambda_\mu(x) \geq \sigma$ outside $V$.
	This completes the proof.
	\end{proof}

\subsection{Proof of Proposition~\ref{cresc_exp}}
\label{proof}

We begin with an outline of the argument. Let $\mathcal{Q}_{\mu,n}(c_0\mu_{f_\mu})$ denote the collection of $n$-cylinders $C(\alpha_1,\dots,\alpha_n)$ such that
\[
\phi_n(\alpha_1,\dots,\alpha_n)\le c_0\mu_{f_\mu}.
\]
Note that $B_n(c_0\mu_{f_\mu}) \subset \mathcal{Q}_{\mu,n}(c_0\mu_{f_\mu})$. Our goal is to estimate
\[
\sum_{C \in \mathcal{Q}_{\mu,n}(c_0\mu_{f_\mu})} \Leb_2(C).
\]

The key difficulty lies in balancing two competing effects: the exponential growth in the number of $n$-cylinders versus the exponential decay of their volumes. We must show that the latter dominates.

To this end, we decompose $\mathcal{Q}_{\mu,n}$ into subfamilies $\mathcal{Q}_{\mu,n,l,t}$ consisting of cylinders whose trajectories visit $R_0$ exactly $t$ times, spending a total of $l$ iterates there. Using the combinatorial bounds from Lemma~\ref{l.combinatorial}, together with the measure estimates from Lemma~\ref{l.medida}, we obtain upper bounds for the total volume of each subfamily (see \eqref{eq_sum_cyl} and \eqref{eq_log_sum_cyl} below).

On the other hand, the defining condition
\[
\phi_n(\alpha_1,\dots,\alpha_n)\le c_0\mu_{f_\mu}
\]
implies that cylinders in $\mathcal{Q}_{\mu,n}$ exhibit weak expansion. Combining this with Lemma~\ref{l.derivada} and the binomial estimates below, we derive constraints on the parameters $l$ and $t$ (Lemma~\ref{l.ruim}). These constraints ensure that the total volume of each $\mathcal{Q}_{\mu,n,l,t}$ is exponentially small. Summing over all admissible pairs $(l,t)$ then completes the proof. Now we fill in the details of this outline.

\begin{lemma}
	\label{l.binomial}
	Given  $\tau > 0$, there exists $\kappa_0 > 0$ such
	that, for every $l \geq 1 $ and $0 < \kappa \leq \kappa_0$,
	$$
	\log \binom{l}{t} \le l(1+\tau)\kappa \log \frac{1}{\kappa}, \quad
	\text{whenever} \quad \; 0 \leq t \leq \kappa l.
	$$
\end{lemma}

\begin{proof}

This Lemma corresponds to \cite[Lemma~4.6]{HV05} with $l_0 = 1$. 
We need to verify that $l_0 =1$ is a valid choice, ensuring that the proof in \cite{HV05} remains applicable.
By Stirling's formula (see, e.g. \cite{Ro55}),
\begin{align*}
	\binom{l}{t} & = \frac{l!}{t!(l-t)!} \le \frac{\sqrt{2\pi l}\; l^le^{-l}(1+1/(4l))} 
	{(\sqrt{2\pi t}\;t^t e^{-t})(\sqrt{2\pi (l-t)}\; (l-t)^{(l-t)} e^{-(l-t)})} .
\end{align*}
Simplifying, this gives
\begin{align*}
	\binom{l}{t}  \leq \frac{1}{\sqrt{t\pi}}  \left(1+\frac{1}{4l}\right) \frac{l^l}{t^t (l-t)^{l-t}} .
\end{align*}
Since for all $l \ge 1$ and every $0 < t < l/2$
$$
\frac{1}{\sqrt{t\pi}} \left(1+\frac{1}{4l}\right) \leq 1,
$$
we have
$$
\binom{l}{t} \leq \frac{l^l}{t^t (l-t)^{l-t}} =
\bigg( \frac{l}{t} \bigg)^t \left( \frac{l}{l-t} \right)^{l-t}
= \left[ \left( \frac{l}{t} \right)^{\frac{t}{l}} \left( \frac{l}{l-t}
\right)^{\frac{l-t}{l}}\right]^l ,
$$
for every $l \geq 1$ and every $0 < t < l/2$.
From this point on the proof follows exactly as in \cite{HV05}. 
For completeness we provide the details.

Since $t \le \kappa l$ is equivalent to
$$
\frac{l}{t} \geq \frac{1}{\kappa} \quad \text{ and } \quad \frac{l}{l-t} \leq
\frac{1}{1-\kappa},
$$
and the function $x^{1/x}$ is an increasing function for $x$ close to $1$ and
it is decreasing for  $x$ large, there exists $\kappa_1 > 0$ such that
$$ \binom{l}{t} \leq \left[ \left( \frac{1}{\kappa}\right)^\kappa
\left( \frac{1}{1 - \kappa}\right)^{1-\kappa}\right]^l \quad \text{for
	every } 0< \kappa \leq \kappa_1. $$

Let $\tau$ be a positive constant.
We define
$$
h(\kappa) = \tau \kappa \log \frac{1}{\kappa} - (1-\kappa) \log
\frac{1}{1-\kappa}.
$$
Then, $h(\kappa)$ is a smooth function for $0 < \kappa < 1$. Moreover,
$$
\lim_{\kappa \rightarrow 0^+} h(\kappa) = \lim_{\kappa \rightarrow 1^-}
h(\kappa)= 0
$$
and $h(\kappa)$ vanishes at some point of the interval $(0,1)$. Furthermore,
the derivative of $h(\kappa)$ is given by
$$
h'(\kappa) = \tau \log \frac{1}{\kappa} + \log \frac{1}{1-\kappa} -1 - \tau .
$$
Therefore, given any $\tau > 0$ there exists $0 < \kappa_0 = \kappa_0(\tau)
\le \kappa_1$ such that $h'(\kappa)> 0$, for every $0< \kappa \leq
\kappa_0$, that is, $h$ is increasing for every $0 < \kappa \leq \kappa_0$.
Hence $h(\kappa) \geq 0$ in this interval. Thus,
$$
(1-\kappa) \log \frac{1}{1-\kappa} \leq \tau \kappa \log \frac{1}{\kappa},
\quad \text{for every } \quad 0<\kappa \le \kappa_0.
$$
Then,
$$
\log \binom{l}{t} \leq l (1+\tau) \kappa \log \frac{1}{\kappa}, \quad
\text{for every } 0 < \kappa \leq \kappa_0.
$$
This proves the lemma.
\end{proof}

\begin{remark}\label{obs2}
	It follows from \eqref{eq.radii} that for each $\mu$ small enough the volume of the hole $H_\mu$, up to a positive constant that depends only on the metric on the center-unstable leaf $W^c(p_\mu)$ is proportional to $\mu$.
	For simplicity of notation, we shall assume $\mu_{f_\mu} = \mu$, which we will do throughout.
\end{remark}

Fix $c_0 = 1/256$. Decompose $\mathcal{Q}_{\mu,n}(c_0\mu) = \bigcup_{l,t} \mathcal{Q}_{\mu,n,l,t}$, as the disjoint union of all $\mathcal{Q}_{\mu,n,l,t} = \mathcal{Q}_{\mu,n}(c_0\mu)\cap \mathcal{C}(n,l,t)$.
By Lemma~\ref{l.derivada}, any $C_n \in \mathcal{C}(n,l,t)$ satisfies
\[
\sum_{j=1}^n \inf_{x\in C_n} \log \|Df_\mu^{-1}(f_\mu^j(x))\|^{-1}
\ge -\frac{3}{32}l\mu - \frac{13}{32}t\log \mu + (n-l)\log \sigma.
\]
Since the left-hand side equals $n\phi_{\mu,n}(\alpha_1,\dots,\alpha_n)$, a necessary condition for $C_n \in \mathcal{Q}_{\mu,n,l,t}$ is
\begin{equation}
\label{e.bad}
-\frac{3}{32}l\mu - \frac{13}{32}t\log \mu + (n-l)\log \sigma
\le c_0 \mu n.
\end{equation}

\begin{lemma}\label{l.ruim}
If $C_n \in \mathcal{Q}_{\mu,n,l,t}$, $n > n_0$ and $\mu >0$ sufficiently small, then
\[
\frac{n-l}{l} \le \frac{1}{8}\frac{\mu}{\log \sigma}
\qquad and \qquad
\frac{t}{l} \le \frac{1}{4}\frac{\mu}{-\log \mu}.
\]
\end{lemma}

\begin{proof}
See \cite[Lemma~4.7]{HV05}
\end{proof}

\medskip

We now estimate the total volume of $\mathcal{Q}_{\mu,n,l,t}$. From Lemmas~\ref{l.combinatorial} and \ref{l.medida},
\begin{equation}
\label{eq_sum_cyl}
\sum_{C\in \mathcal{Q}_{\mu,n,l,t}} \Leb_2(C)
\le \binom{l}{t-1} \binom{n-l}{t-1}
\left(\frac{\eta}{\sigma^2}\right)^{n-l}
\prod_{i=1}^t \Jac f_\mu^{l_i}(x_i)^{-1}.
\end{equation}

The second term in the right hand side of \eqref{eq_sum_cyl} can be estimate using
\begin{equation}
	\label{eq.elementary}
	\binom{m}{l} \le \sum_{j=0}^{m} \binom{m}{j} = 2^m, \quad 0 \le t \le m,
\end{equation}
and for the last term we use part (b) of Lemma~\ref{l.jac_local} to get
\begin{equation}
	\label{eq.contraction}
	\log \prod_{i=1}^t \Jac f_\mu^{l_i} (x_i)^{-1} \le - \frac{61}{32}\mu l.
\end{equation}
We bound the first term as follows.

\begin{lemma}
\label{l.binom}
There exists $\mu_0 > 0$ such that for every $\mu < \mu_0$, $n > n_0$, and $C_n \in \mathcal{Q}_{\mu,n,l,t}$,
\[
\log \binom{l}{t-1} \le \frac{13}{32}\mu l.
\]
\end{lemma}

\begin{proof}
	The second inequality of Lemma~\ref{l.ruim}, establishes that for $n$-cylinder in $\Q_{\mu,n,l,t}(c_0\mu_{f_\mu})$ the values of $l$ and $t$ must satisfy a relation that depends on $\mu$\;:
	$$
	t \le \frac{\mu}{-4 \log \mu} l.
	$$
	We write $\kappa = \mu/(-4\log \mu)$. Notice that
	$$
	\frac{1}{\mu}\kappa \log \frac{1}{\kappa} = \frac{ \log(-4 \log \mu)}{-4 \log \mu} - \frac{\log \mu}{-4 \log \mu}
	$$
	converges to $1/4$ when $\mu$ goes to zero.
	Now, we fix, $\tau = 1/1000$. 
	Then, there exists $0 < \mu_0 \le \mu_1$ such that for every $0 < \mu <\mu_0$ 
	$$
	\frac{1}{\mu}\kappa \log \frac{1}{\kappa}  \le \frac{13}{32(1+\tau)}.
	$$
	Applying Lemma~\ref{l.binomial}, we obtain
	$$
	\log \binom{l}{t-1} \le \log \binom{l}{t} \le l(1+\tau)\kappa \log \frac{1}{\kappa} \le \frac{13}{32} \mu l.
	$$
	This completes the proof.
\end{proof}

Combining these estimates, we obtain
\begin{equation}
\label{eq_log_sum_cyl}
\log \sum_{C\in \Q_{\mu,n,l,t}} \Leb_2(C) \le \frac{13}{32} \mu l + (n-l) \log 2 + (n-l) \log \frac{\eta}{\sigma^2} - \frac{61}{32}\mu l.
\end{equation}

Using $\eta \le 1000\sigma^2$ and Lemma~\ref{l.ruim},
\[
(n-l)\log 2 + (n-l)\log \frac{\eta}{\sigma^2}
\le (n-l)\log 2000 \le \mu l,
\]
for $\mu$ small. Hence,
\begin{equation}
\label{eq.taquase}
\log \sum_{C\in \mathcal{Q}_{\mu,n,l,t}} \Leb_2(C)
\le -\frac{1}{2}\mu l \le -\frac{1}{4}\mu n.
\end{equation}

Finally, summing over all admissible $l$ and $t$ (at most $\frac{\mu}{-4\log\mu}n^2$ terms, since $l \le n$ and $t \le \frac{\mu}{-4\log\mu}$), we obtain
\[
\sum_{C \in \mathcal{Q}_{\mu,n}(c_0\mu)} \Leb_2(C)
\le \frac{\mu}{-4 \log \mu} n^2 e^{-(1/4)\mu n}.
\]
Since $B_n(c_0\mu) \subset \mathcal{Q}_{\mu,n}(c_0\mu)$, this yields
\[
\Leb_2(B_n(c_0\mu)) \le \frac{\mu}{-4 \log \mu} n^2 e^{-(1/4)\mu n}.
\]

This completes the proof of Proposition~\ref{cresc_exp}.


\end{document}